\documentclass[letterpaper]{amsart}
\usepackage[all]{xy}
\usepackage{amsrefs}
\usepackage{amsthm}
\usepackage{amssymb}
\usepackage{amsmath}
\usepackage[CJKbookmarks=true]{hyperref}
\title{A Geometric Realization of Symmetric Pairs of Type AIII}
\author{Rui Xiong}
\address{Academy of Mathematics and Systems Science, Chinese Academy of Sciences, Beijing, China}
\email{XiongRui\_Math@126.com}

\newtheorem{Th}{Theorem}
\newtheorem{Lemma}[Th]{Lemma}
\newtheorem{Coro}[Th]{Corollary}
\newtheorem{Prop}[Th]{Proposition}

\newtheorem{Eg}[Th]{Example}
\newtheorem{Rmk}[Th]{Remark}

\begin{document}

\def\Hom{\operatorname{Hom}}
\def\Ext{\operatorname{Ext}}
\def\Fun{\mathsf{Fun}}
\def\vv{\mathbf{v}}
\def\ww{\mathbf{w}}
\def\im{\operatorname{im}}
\def\Gr{\operatorname{\mathcal{G}\!\it r}}

\maketitle

\begin{abstract}
In this paper, we study the fixed loci of Nakajima quiver variety $\mathfrak{L}$ under an involution, which is related to the $\sigma$-quiver variety introduced in \cite{Li18}. 
We give a geometric realization of symmetric pair of type AIII over $\sigma$-quiver varieties as an analogy of the construction over Nakajima quiver varieties in \cite{Nakajima94}. 
It partially recovers the construction of sympletic partial flag varieties studied in \cite{BaoKujwaLiWang}. 
This gives an affirmative answer to a variation of a conjecture raised in \cite{Li18} for type AIII. 
To achieve this, we define the $\iota$-analogy of Hecke correspondences and study the properties of their fibres and products. 
\end{abstract}

\section{Introduction}

Let $\mathbf{U}(\mathfrak{g})$ be the universal enveloping algebra of a simple Lie algebra $\mathfrak{g}$. 
Let $\mathbf{U}_q(\mathfrak{g})$ be its quantization. In \cite{BLM1990}, the authors constructed an algebra homomorphism
\begin{equation}
\mathbf{U}_q(\mathfrak{sl}_n)\longrightarrow \operatorname{End}(\Fun(\mathcal{F}))
\end{equation}
where $\Fun(\mathcal{F})$ is the space of functions over partial flag varieties $\mathcal{F}$ of length $n$ of a vector space over finite fields. 
This action is very instructive due to the following reasons
\begin{itemize}
\item 
The weight decomposition of $\Fun(\mathcal{F})$ under this action corresponds to the decomposition of $\mathcal{F}$ into connected components, i.e. the dimension vectors of the subspaces in the flags. 
\item 
The Chevalley generators of positive/negative part of $\mathbf{U}_q(\mathfrak{sl}_n)$, up to some power of $q$, are given by ``Hall-algebra-like'' correspondences, i.e. 
``adding/lowering the flag by one dimension''. 
\item 
It is essentially a geometric realization of the representation of $N$-th symmetric power of natural representation where $N$ is the dimension of the ambient vector space. 
\end{itemize}
Their argument works over $\mathbb{C}$ either, where it gives a Lie algebra homomorphism 
$$\mathfrak{sl}_n\longrightarrow \operatorname{End}(\Fun(\mathcal{F}))$$
where $\Fun(\mathcal{F})$ is the space of constructible functions over partial flag varieties $\mathcal{F}$ of length $n$ of a vector space over $\mathbb{C}$. 

Later, Nakajima \cite{Nakajima94} gave a much more general construction for all symmetric Kac--Moody algebra and dominant weights $\ww$. To be exact, Nakajima constructed a Lie-algebra homomorphism
\begin{equation}
\mathfrak{g}\longrightarrow \operatorname{End}(\Fun(\mathfrak{L}(\ww)))
\end{equation}
where $\Fun(\mathfrak{L}(\ww))$ is the space of constructible functions over Nakajima quiver variety $\mathfrak{L}(\ww)$. 
Similarly to the construction of \cite{BLM1990}, the weight spaces correspond to the connected components, and the action is ``Hall algebra-like'', see Theorem \ref{MainThNak94}. 
It is a geometric realization of the representation of highest weight $\ww$. 
Moreover, this construction recovers the construction of \cite{BLM1990} for $\mathfrak{sl}_n$ at $q=1$ and $\ww$ is the highest weight for the symmetric power of natural representation, see Example \ref{ExampleofNakQV}. 
The correspondence is known as Hecke correspondences (after Nakajima \cite{Nakajima94} and \cite{Nakajima98}). 

From the theory of real simple Lie algebras, it is natural to consider \emph{symmetric pairs} $(\mathfrak{g},\mathfrak{k})$ where $\mathfrak{k}$ is the fixed subalgebra of $\mathfrak{g}$ under some involution.
It has a quantization $(\mathbf{U}_q(\mathfrak{g}),\mathbf{U}^\iota_q(\mathfrak{k}))$, known as the \emph{quantum symmetric pair} \cite{Letzter1999} where $\mathbf{U}^\iota_q$ is also known as the \emph{$\iota$quantum group}. 
The theory is recently strongly developed by works of Wang and his colleagues, see \cite{wang2021quantum} for a survey. 

Consider the symmetric pair $(\mathfrak{g},\mathfrak{k})$ of type AIII induced by the longest word of the Weyl group of type $A_{\text{odd}}$. 
It can be illustrated by the following diagram
$$
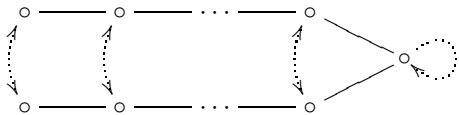
\xymatrix@!=1pc{
\circ\ar@{-}[r]
\ar@{<..>}@/_/[d]&
\circ\ar@{-}[r]
\ar@{<..>}@/_/[d]&
\cdots\ar@{-}[r]&
\circ
\ar@{}[r];[dr]|{{\displaystyle\,\,\,\circ\,}\phantom{0}}="zero"
\ar@{-}"zero"\ar@{<..>}@/_/[d]&\ar@{..>}@(ur,dr)"zero";"zero"
\\
\circ\ar@{-}[r]&
\circ\ar@{-}[r]&
\cdots\ar@{-}[r]&
\circ\ar@{-}"zero" & 
}$$
In \cite{BaoKujwaLiWang}, the authors constructed an algebra homomorphism 
\begin{equation}
\mathbf{U}^\iota_q(\mathfrak{k})\longrightarrow \operatorname{End}(\Fun(\mathcal{F}^\iota))
\end{equation}
where $\mathcal{F}^\iota$ is the space of functions over partial flag varieties of classical types over finite fields. 

The main result of this paper is to construct an $\iota$-analogy of the Nakajima picture. 
To be exact, we will construct a Lie-algebra homomorphism 
\begin{equation}
\mathfrak{k}\longrightarrow  \operatorname{End}(\Fun(\mathfrak{R}(\ww)))
\end{equation}
for symmetric pair of type AIII, 
where $\Fun(\mathfrak{R}(\ww))$ is the space of constructible functions over the $\sigma$-quiver variety $\mathfrak{R}(\ww)$. See Theorem \ref{MainTh} for the detailed description. 

Our construction has the following features
\begin{itemize}
\item We have a decomposition $\mathfrak{R}(\ww)=\bigsqcup \mathfrak{R}(\vv,\ww)$ which corresponds to the weight decomposition. 
\item The Chevalley generators act by ``Hall-algebra-like'' correspondences, which we will refer them as \emph{$\iota$Hecke correspondences}. 
\item The partial flag varieties $\mathcal{F}^\iota$ of type $C$ appears as a special case of $\mathfrak{R}(\ww)$, see Example \ref{ExampleofsigmaNakQV}. 
\end{itemize}
Since the representation theory for symmetric pair is not clear at the present stage, we cannot describe which representation our geometric realization gives. 

In \cite{Li18}, the $\sigma$-quiver variety $\mathfrak{S}(\ww)$ was introduced to give an $\iota$-analogy of Nakajima's picture.
Our $\sigma$-quiver variety $\mathfrak{R}(\ww)$ is a closed subvariety of the $\sigma$-quiver $\mathfrak{S}(\ww)$, which is parallel to the Nakajima quiver variety $\mathfrak{L}(\ww)$ and $\mathfrak{M}(\ww)$, see Remark \ref{Rmkonsigmavariety} and Theorem \ref{sigmacoincides}. 
It was conjectured that there is a Lie algebra homomorphism 
\begin{equation}
\mathfrak{k}\longrightarrow  
H_{\text{top}}^{\textsf{BM}}(\mathfrak{Z}(\ww))
\end{equation}
where $H^{\textsf{BM}}_{\text{top}}(\mathfrak{Z}(\ww))$ is the algebra of top degree part of Borel--Moore homology of certain Steinberg variety under convolution product. This is the $\iota$-analogy to Borel--Moore homology construction in \cite{Nakajima98}. 
The special cases of partial flag variety of classical types are checked recently in \cite{Li21Quasispli} which is parallel to the Springer theory of $\mathbf{U}(\mathfrak{sl}_n)$, see \cite[Chapter 4]{CG1997}. 
Our result can be viewed as a variation of this conjecture. 
The existence of such homomorphism can also be viewed as evidence of this conjecture. 

With the development of geometric representation theory, it is noted that geometric realization usually helps us understand the algebra itself. 
But the problem of geometric realization of (quantum) symmetric pair is still very open, see \cite[Section 9 (2)]{wang2021quantum}. From our proof, we see that similar construction for $\sigma$-quiver variety might be difficult for other types. 
Due to the private discussion with Wang, the reason would be that the $\sigma$-quiver varieties are defined by Vogan diagrams rather than Satake diagrams. 
It is worth mentioning that there are other geometric ways to realize quantum symmetric pairs, for example \cite{LuWang2021}. It is closely related to the Hall algebra construction \cite{LuWang2019}. 
Though there is no evidence that it is related to the partial flag varieties of type B or C of \cite{BaoKujwaLiWang} in a geometric way. 

Last but no mean least, to achieve an action of (affine) $\iota$quantum group, it is plausible to use equivariant K-theory as the $\iota$-analogy of \cite{Nakajima2000}. 
Note that the Faddeev--Reshetikhin--Takhtajan construction of symmetric pair of Yangian was achieved in \cite{Li18}, which is parallel to the construction of \cite{MaulikOkounkovQuantum}. 
But it cannot be moved to equivariant K-theory directly, since we do not know whether a K-theoretic stable envelope exists over $\sigma$-quiver variety. 


\subsection*{Acknowledgment.} 
The author thanks Yiqiang Li, Ming Lu, Quan Situ, Weiqiang Wang and Yehao Zhou for the discussion. The author would express his gratitude to Weiqiang Wang for his nice lectures on $\iota$program. When preparing this paper, the author was visiting the Academy of Mathematics and Systems Science, the Chinese Academy of Science. The author would be grateful for the invitation of Si'an Nie and the nice hospitality.


\section{Nakajima quiver varieties}

Let $Q$ be the Dynkin diagram of type $A_{2d-1}$ labelled as follows 
$$\xymatrix{
\underset{\hspace{-2pc}1\hspace{-2pc}}\circ\ar@{-}[r]&
\underset{\hspace{-2pc}2\hspace{-2pc}}\circ\ar@{-}[r]&
\cdots\ar@{-}[r]&
\underset{\hspace{-2pc}2d-2\hspace{-2pc}}\circ\ar@{-}[r]&
\underset{\hspace{-2pc}2d-1\hspace{-2pc}}\circ
}$$
Denote $\mathbb{I}=\{1,\ldots,2d-1\}$ the index set of vertices. 
Let $(c_{ij})$ be the Cartan matrix. 
That is, 
$$c_{ij}=\begin{cases}2, & i=j,\\ -1, & |i-j|=1,\\ 0,& \text{otherwise}.\end{cases}$$

We construct a new quiver $Q^{\textsf{fr}}$ as follows.
The vertices of $Q^{\textsf{fr}}$ is $\mathbb{I}\cup \mathbb{I}_{\mathsf{fr}}$ with $\mathbb{I}_{\textsf{fr}}$ a copy of $\mathbb{I}=\{\underline{i}:i\in\mathbb{I}\}$. 
For unframed vertices, we have an arrow $i\to j$ if there is an edge in $Q$. Besides, we have an arrow from $i\to \underline{i}$ for each $i\in \mathbb{I}$. Here is an example of $A_3$.
$$
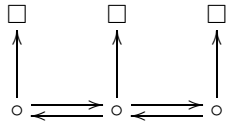
\xymatrix{\square&\square&\square\\
\circ\ar[u]\ar@<+0.5ex>[r]&\ar@<+0.5ex>[l]\circ\ar[u]\ar@<+0.5ex>[r]&\ar@<+0.5ex>[l]\circ\ar[u]}$$
For any arrow $h:i\to j$, denote $\bar{h}:j\to i$ the reversed arrow. 
Let $\mathcal{L}$ be the $\mathbb{C}$-linear category generated by the quiver $Q^{\textsf{fr}}$ with \emph{commutative relations} for each $i\in \mathbb{I}$, 
\begin{equation}
\sum_{h:i\to j} \pm \bar{h}h=0, \text{ with } \pm=\begin{cases}
1, & j=i+1,\\
-1, & j=i-1. 
\end{cases}\label{commutativeRelation}
\end{equation}
This is in principle the category introduced in \cite{KellerScherotzke13}. 
Denote $\mathcal{Q}$ (resp. $\mathcal{S}$) be the full subcategory generated by unframed vertices (resp. framed vertices). 
Note that $\mathcal{S}$ is discrete.

\def\ddim{\operatorname{\mathbf{dim}}}
A \emph{framed representation} is a $\mathbb{C}$-linear functor from $\mathcal{L}$ to $\mathbb{C}\text{-}\mathsf{Vec}$ the category of $\mathbb{C}$-vector spaces. 
Morphisms of framed representations are defined to be natural transforms. 
We say a representation $F$ is \emph{stable} if 
for any representation $G$ 
\begin{equation}
G|_{\mathcal{S}}=0 \Longrightarrow \Hom(G,F)=0.
\end{equation}
Let $\ww=(w_i)_{i\in \mathbb{I}}$ and $\vv=(v_i)_{i\in \mathbb{I}}$ be two dimension vectors. 
The \emph{Nakajima quiver variety} $\mathfrak{L}(\vv,\ww)$ is the moduli space of stable framed representations $F$ with 
\begin{equation}
\ddim F|_{\mathcal{Q}}=\vv,\text{ and }
\ddim F|_{\mathcal{S}}=\ww.
\end{equation}
where $\ddim X\in \mathbb{N}^{\mathbb{I}}$ is the dimension vector for any $\mathbb{I}$-graded vector space $X$.
This coincides with Nakajima's original definition \cite{Nakajima94}, see \cite{KellerScherotzke13} for the proof. 
We will write $\mathfrak{L}(\ww)=\bigsqcup \mathfrak{L}(\vv,\ww)$. 

Assume we are given a functor $W$ from $\mathcal{S}$ to  $\mathbb{C}\text{-}\mathsf{Vec}$, i.e. an $\mathbb{I}$-graded vector space $(W(i))_{i\in \mathbb{I}}$. 
We can define a stable framed representation $K_RW$ with $K_RW|_{\mathcal{S}}=W$ as follows. 
For each unframed vertice $i$, 
\begin{equation}
K_RW(i)=\bigoplus_{j\in \mathbb{I}}
\Hom_{\mathbb{C}}(\mathcal{Q}(i,j),W(j)).
\end{equation}
For each arrow $h:i\to i'$ in $\mathcal{Q}$, 
\begin{equation}
K_RW(h)= \bigg[(f_j)\longmapsto (f_j\circ \mathcal{Q}(h,j))\bigg]: K_RW(i)\longrightarrow K_RW(j).
\end{equation}
For each arrow $\alpha:i\to \underline{i}$, 
\begin{equation}
K_RW(\alpha)= \bigg[(f_j)\longmapsto f_i(\text{id}_i)\bigg]: K_RW(i)\longrightarrow W(i). 
\end{equation}
For any representation $F$ with $F|_{\mathcal{S}}=W$, there is morphism of framed representation $\epsilon:F\to K_RW$ given by evaluation, say
\begin{equation}
\epsilon(i)(v)=(f_j)\text{ with } f_j(p)=\alpha_j(p(v)),
\end{equation}
where $\alpha_j:F(j)\to W(j)$ the morphism corresponding to $j\to \underline{j}$.

Actually, $K_R$ is the right adjoint functor (so-called right Kan extension, see \cite{KellerScherotzke13}) of restriction  functor from $\mathcal{L}$ to $\mathcal{S}$.

\begin{Prop}[Keller and Scherotzke
{\cite{KellerScherotzke13}}]\label{LissubofKRW}
The representation $F$ is stable if and only if the natural map $\epsilon: F\to K_RW$ is injective. 
In particular, $\mathfrak{L}(\vv,\ww)$ is the quiver Grassmannian of $K_RW|_{\mathcal{Q}}$ for any $W$ such that $\ddim W=\ww$. 
\end{Prop}

\begin{Prop}\label{dimKRW}
Let $\vv'=\ddim K_RW|_{\mathcal{Q}}$. 
Then $w_0\ww=\ww-C\vv'$ for the longest element $w_0$ of Weyl group. 
\end{Prop}

\begin{proof}
By \cite{Nakajima98}, $\mathfrak{L}(\vv',\ww)$ corresponds to the space of lowest weight $\ww-C\vv'$ of the representation of highest weight $\ww$. 
\end{proof}

\begin{Eg}\label{ExampleofNakQV} Consider the following diagram
$$
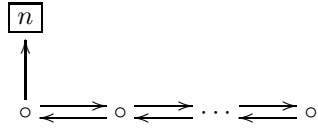
\xymatrix@!=1pc{
\fbox{$n$}\\
\circ\ar[u]\ar@<+0.5ex>[r]&\ar@<+0.5ex>[l]\circ\ar@<+0.5ex>[r]&\ar@<+0.5ex>[l]\cdots\ar@<+0.5ex>[r]&\ar@<+0.5ex>[l]\circ
}$$
That is, $\ww$ has only one nonzero component $n$ at the left end of Dynkin diagram of type $A$. 
Let $W$ be an vector space of dimension $n$. 
Then $K_RW$ is the representation
$$
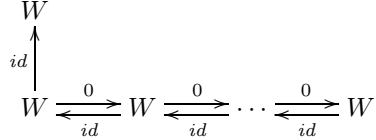
\xymatrix{
W\\
\ar[u]^{\text{id}}W\ar@<+0.5ex>[r]^{0}&
\ar@<+0.5ex>[l]^{\text{id}} W\ar@<+0.5ex>[r]^{0}&
\ar@<+0.5ex>[l]^{\text{id}} \cdots \ar@<+0.5ex>[r]^{0}&
\ar@<+0.5ex>[l]^{\text{id}} W 
}$$
The subrepresentations of $K_RW|_{\mathcal{Q}}$ is simply the partial flags of $W$. 
In particular, $\mathcal{L}(\ww)$ is the variety of partial flags of $W$.  
\end{Eg}

\begin{Eg}\label{SmallEgA}Consider the following diagram 
$$
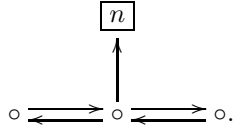
\xymatrix{
&\fbox{$n$}\\
\circ\ar@<+0.5ex>[r]&\ar@<+0.5ex>[l]
\circ\ar[u]\ar@<+0.5ex>[r]&\ar@<+0.5ex>[l]
\circ.
}$$
That is, $\ww$ has only one nonzero component $n$ at the middle vertex of Dynkin diagram $A_3$. 
Let $W$ be a vector space of dimension $n$. Now
$K_RW$ is exactly 
$$
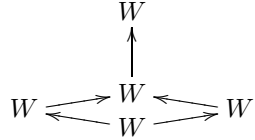
\xymatrix{
&W\\
\ar@{}[rr]^{\displaystyle W}="uW"_{\displaystyle W}="dW"
W\ar"uW"\ar"dW";[]&\ar"uW";[u]
&W\ar"uW"\ar"dW";[],\\
}$$
where all illustrated arrows are identities, and other arrows (not figured out) are zero. 
To be exact, $(K_RW)(i)=W\oplus W$ for the middle vertex $i$ and $(K_RW)(i)=W$ for the left or right vertices $i$. 
Actually, this representation achieves the dimension predicted in Proposition \ref{dimKRW}, and it is stable, thus it is exactly $K_RW$. 
In particular, when $n=1$, each nonempty $\mathfrak{L}(\vv,\ww)$ is a point. 
\end{Eg}

\begin{Eg}\label{SmallEgB}Consider the following diagram 
$$
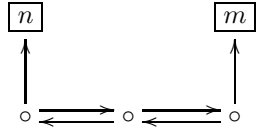
\xymatrix{
\fbox{$n$}&&\fbox{$m$}\\
\circ\ar[u]\ar@<+0.5ex>[r]&\ar@<+0.5ex>[l]
\circ\ar@<+0.5ex>[r]&\ar@<+0.5ex>[l]
\circ\ar[u].}$$
Let $N$ (resp. $M$) be a vector space of dimension $n$ (resp. $m$). Then $K_RW$ can be representated by the following diagram
$$
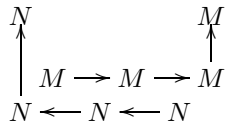
\xymatrix{
\ar@{}[r]|>>{\displaystyle N}="tN"
&&&\ar@{}[r]|<<{\displaystyle M}="tM"
&\\
\ar@{}[r]_>>{\displaystyle N}="aN"
&\ar@{}[r]^<<{\displaystyle M}="aM"
\ar@{}[r]_>>{\displaystyle N}="bN"
&\ar@{}[r]^<<{\displaystyle M}="bM"
\ar@{}[r]_>>{\displaystyle N}="cN"
&\ar@{}[r]^<<{\displaystyle M}="cM"&
\ar"aN";"tN"\ar"bN";"aN"\ar"cN";"bN"
\ar"cM";"tM"\ar"aM";"bM"\ar"bM";"cM"
}$$
Actually, this representation achieve the dimension predicted in Proposition \ref{dimKRW}, and it is stable, thus it is exactly $K_RW$.
In particular, when $n=1$, and $\vv=(1,1,1)$, $\mathfrak{L}(\vv,\ww)$ is a union of three copies of $\mathbb{P}^1$ with intersection diagram $A_3$, see \cite{Nakajim1999}. 
\end{Eg}

\section{Geometric realization of $\mathbf{U}$}

Let $\mathfrak{g}$ be the complex simple Lie algebra correspondent to $Q$. 
By a theorem of Serre, for example \cite{Knapp2002}, $\mathfrak{g}$ is the Lie algebra generated by $E_i$, $F_i$ and $H_i$ for $i\in \mathbb{I}$ with relations 
\begin{gather}
[H_i,H_j]=0, \quad 
[H_i,E_j]=c_{ij}E_j,\quad
[H_i,F_j]=-c_{ij}F_j\\
[E_i,F_j]=\delta_{ij}H_i\\
[E_i,E_j]=[F_i,F_j]=0\quad \text{ if $c_{ij}=0$}\\
[E_i, [E_i,E_i]]=[F_i, [F_i,F_i]]=0 \quad \text{ if $c_{ij}=-1$}
\end{gather}
Here $\delta_{ij}$ stands the Kronecker's delta. 
Denote $\mathbf{U}$ the universal enveloping algebra of $\mathfrak{g}$. 
Nakajima \cite{Nakajima94} constructed a $\mathbf{U}$-action using the constructible functions $\Fun(\mathfrak{L}(\ww))$ over $\mathfrak{L}(\ww)$. 
We will shortly review the theory of the basic properties of such space. 

Let $X$ be a variety over $\mathbb{C}$. 
Denote $\chi(X)$ the \emph{Euler characteristic} of $X$, that is
\begin{equation}
\chi(X)=\sum_n(-1)^n\dim H^n_c(X;\mathbb{Q}),
\end{equation}
where $H_c^n(X;\mathbb{Q})$ is the cohomology of compact support. The following properties are well-known. 

\begin{Prop} We have the following properties of $\chi$. 
\begin{itemize}
    \item For a closed subvariety $F\subseteq X$, with complement $U=X\setminus F$
    \begin{equation}
    \chi(F)+\chi(U)=\chi(X).\label{additionprinciple}
    \end{equation}
    \item For a fibre bundle $X\to B$ with fibre $F$
    \begin{equation}
    \chi(F)\chi(B)=\chi(X).
    \end{equation}
\end{itemize}
\end{Prop}
Let $\Fun(X)$ be the space of constructible functions with value in $\mathbb{Z}$ over algebraic variety $X$. 
For any morphism $f:X\to Y$, we have a natural \emph{pull back} $f^*:\Fun(Y)\to \Fun(X)$ by 
\begin{equation}
f^*\psi=\psi\circ f.
\end{equation}
We can define \emph{push forward}
$f_*:\Fun(X)\to \Fun(Y)$  by 
\begin{equation}
(f_*\varphi)(y)=\sum_{n}\chi\bigg(\{x\in X:f(x)=y,\varphi(x)=n\}\bigg)n.
\end{equation}
with $n$ going through all values of image $\psi$.

The followng proposition is standard. 
\begin{Prop}
These functors have the following properties 
\begin{itemize}
\item For any morphism $f:X\to Y$ and $\psi_\bullet\in \Fun(Y)$ for $\bullet=1,2$ 
\begin{equation}
f^*\psi_1\cdot f^*\psi_2=f^*(\psi_1\cdot \psi_2).
\end{equation}
\item For any morphism $f:X\to Y$, $\varphi\in \Fun(X)$ and $\psi\in \Fun(Y)$, 
\begin{equation}
f_*(\varphi\cdot f^*\psi)=f_*\varphi\cdot \psi.
\end{equation}
\item Let $i:X\to Y$ be an inclusion of subvariety. For any $\psi\in \Fun(Y)$, 
\begin{equation}
i_*i^*\varphi = \mathbf{1}_{X}\cdot \varphi
\end{equation}
where $\mathbf{1}_X$ is the characteristic function of $X$. 
\item Let $f:X\to Y$ be any morphism. We have 
\begin{equation}
f_*\mathbf{1}_X(y)=\chi(f^{-1}(y)). 
\end{equation}
\item 
Assume we have a Cartesian square
\begin{equation}
\begin{array}{@{}c@{}}\xymatrix{
X'\ar[d]_{f'}\ar[r]^q& X\ar[d]^f\\
Y'\ar[r]_p & Y.}\end{array}
\end{equation} 
Then, for any $\varphi\in \Fun(X)$, 
\begin{equation}
f'_*q^*\varphi=p^*f_*\varphi\label{Cartesiansquare}
\end{equation}
\end{itemize}
\end{Prop}

Now, let us turn to Nakajima quiver variety $\mathfrak{L}(\ww)$.
Denote $\mathbf{e}_j=(\delta_{ij})_{j\in \mathbb{I}}$ with $\delta_{ij}$ the Kronecker's delta. 
Let $\mathfrak{P}_i$ be the variety of pairs $(F_1,F_2)$ of representations in $\mathfrak{L}(\ww)$ such that $F_1\subseteq F_2$ with $\ddim F_2/F_1=\mathbf{e}_i$. 
The variety $\mathfrak{P}_i$ is called the \emph{Hecke correspondence} (after \cite{Nakajima94} and \cite{Nakajima98}). 
We have two projections $\rho_{i\bullet}:\mathfrak{P}_i\to \mathfrak{L}(\ww)$ for $\bullet=1,2$. 
$$\xymatrix{&\mathfrak{P}_i\ar[dl]_{\rho_{i1}}\ar[dr]^{\rho_{i2}}\\
\mathfrak{L}(\ww)&&\mathfrak{L}(\ww).}$$

For $F\in \mathfrak{L}(\vv,\ww)$ and $i\in\mathbb{I}$, denote $C_i(F)$ the complex of vector spaces, 
$$0\longrightarrow\, \stackrel{-1}{\rule{0pc}{1.5pc}F(i)}\,\longrightarrow\, \stackrel{0}{\rule{0pc}{1.5pc}W(i)\oplus\bigoplus_{h:i\to j}F(j)}\,\longrightarrow\,
\stackrel{1}{\rule{0pc}{1.5pc}F(i)}\,\longrightarrow 0.$$

\begin{Prop}[Nakajima {\cite{Nakajima94}}]\label{Importantcomplex}
We have
\begin{gather}
H^{-1}(C_i(F))=0,\\
\mathbb{P}(H^0(C_i(F)))\cong \rho_{i1}^{-1}(F),\\ \mathbb{P}(H^1(C_i(F))^*)\cong \rho_{i2}^{-1}(F), 
\end{gather}
where $\mathbb{P}(-)$ is the associated projective space of a vector space. 
\end{Prop}

Note that $\rho_{i1}^{-1}(F)$ (resp. $\rho_{i2}^{-1}(F)$) is the variety of sup-representation (resp. subrepresentation) $F'$ of $F$ with dimension difference $\mathbf{e}_i$. 

Let 
\begin{gather}
\phi_i(F)=\dim H^0(C_i(F))=\chi(\rho_{i1}^{-1}(F)),\\
\epsilon_i(F)=\dim H^1(C_i(F))=\chi(\rho_{i2}^{-1}(F)).
\end{gather}
It is clear that $\phi_i(F)-\epsilon_i(F)=(\ww-C\vv)_i$ for $\vv=\ddim F$.

We know $\mathfrak{L}(\ww)=\bigsqcup \mathfrak{L}(\vv,\ww)$, thus $\Fun(\mathfrak{L}(\ww))=\bigoplus_{\vv} \Fun(\mathfrak{L}(\vv,\ww))$. 
For any function $f$ of $\vv$, it defines a map $f:\Fun(\mathfrak{L}(\ww))\to \Fun(\mathfrak{L}(\ww))$ by 
\begin{equation}
f(\varphi_{\vv})= \big(f(\vv) \varphi_{\vv}\big).
\end{equation}

\begin{Th}[Nakajima {\cite{Nakajima94}}]\label{MainThNak94}
There is a well-defined action of $\mathbf{U}$-action over $\Fun(\mathfrak{L}(\ww))$ by 
\begin{gather}
H_i\longmapsto (\ww-C\vv)_i \\
E_i\longmapsto \rho_{i1*}\rho_{i2}^*\\
F_i\longmapsto \rho_{i2*}\rho_{i1}^*
\end{gather}
\end{Th}

\begin{Eg}Especially, the Example \ref{ExampleofNakQV} recovers the Beilinson--Lusztig--MacPherson construction of ${\mathbf{U}}_q(\mathfrak{sl}_n)$ in \cite{BLM1990} at $q=1$.
\end{Eg}

\section{$\sigma$-quiver varieties}

Denote $\sigma$ the involution over $Q$ induced by the longest element $w_0$ in the Weyl group, that is, $\sigma(i)=2d-i$. 
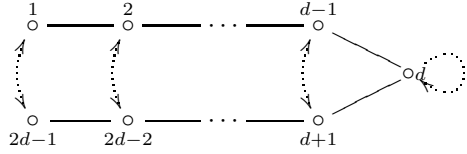
$$\xymatrix@!=1pc{
\stackrel{1}\circ\ar@{-}[r]
\ar@{<..>}@/_/[d]&
\stackrel{2}\circ\ar@{-}[r]
\ar@{<..>}@/_/[d]&
\cdots\ar@{-}[r]&
\stackrel{\!\!\!\!d-1\!\!\!\!}\circ
\ar@{}[r];[dr]|{\phantom{00} {\displaystyle\circ} d\phantom{00}}="zero"
\ar@{-}"zero"\ar@{<..>}@/_/[d]&\ar@{..>}@(ur,dr)"zero";"zero"
\\
\underset{\!\!\!\!2d-1\!\!\!\!}\circ\ar@{-}[r]&
\underset{\!\!\!\!2d-2\!\!\!\!}\circ\ar@{-}[r]&
\cdots\ar@{-}[r]&
\underset{\!\!\!\!d+1\!\!\!\!}\circ\ar@{-}"zero" & 
}$$
This is known as \emph{type AIII}. 

\begin{Lemma}\label{SymmetricPairing}There is a natural non-degenerate pairing 
\begin{equation}
B_{ij}:\mathcal{Q}(i,j)\otimes \mathcal{Q}(\sigma(i),j)\to \mathbb{C},
\end{equation}
such that 
\begin{gather}\label{adjointofpairingB}
B_{ij}(p,q)=B_{\sigma(i),j}(q,p),\\
B_{i',j}(p\circ h, q)=B_{\sigma(i'),j}(p,q\circ \sigma(\bar{h}))
\end{gather}
for any $p\in \mathcal{Q}(i,j)$, $q\in \mathcal{Q}(\sigma(i),j)$ and $h\in \mathcal{Q}(i,i')$. 
\end{Lemma}
\begin{proof}
By a simple combinatorial argument, there is a unique path $[d_i]$ of length $|\mathbb{I}|-1=2d-2$ from $i$ to $\sigma(i)$ up to commutative relations (\ref{commutativeRelation}). 
Let us denote 
$B_{ij}(p,q)$ to be the coefficient of $[d_j]$ in $\sigma(q)\overline{p}$. It is clear that this pairing satisfies the properties in the assertion. 
\end{proof}

\begin{Rmk}This pairing is in principle the Auslander--Reiten formulae
$$\operatorname{Hom}_{\vec{Q}}(U,V)^*\cong \operatorname{Ext}_{\vec{Q}}(V,\tau U),$$
for an orientation $\vec{Q}$ of $Q$, where $\tau$ is the Auslander--Reiten translation, see for example \cite{kirillov2016quiver}. 
But for other Dynkin types, the pairing is not explicit as that for type A, and the signs would not be complicated. 
\end{Rmk}

Let $W=(W_i)_{i\in \mathbb{I}}$ be an $\mathbb{I}$-graded symplectic vector space. 
\begin{Coro}\label{symplecticformoverKRW}
There is a symplectic form over $\bigoplus_{i\in \mathbb{I}}K_RW(i)$ induced by 
\begin{equation}
\omega: K_RW(i)\otimes K_RW(\sigma(i))\longrightarrow \mathbb{C}
\end{equation}
such that 
\begin{gather}
\omega(x,y)=-\omega(y,x),\\ \omega(h(x),y)=\omega(x,\sigma(\bar{h})y)
\end{gather}
for any arrows $h$ in $K_RW$. 
\end{Coro}
\begin{proof}
By identifying $\Hom(\mathcal{Q}(i,j),W(j))=
\mathcal{Q}(i,j)^*\otimes W(j)$, the form is given by
$$
\omega(p\otimes x,q\otimes y)=B_{ij}^{\text{t}}(p,q)\omega(x,y)\quad 
\begin{array}{c}
p\in \mathcal{Q}(i,j)^*,x\in W(j)\\
q\in \mathcal{Q}(\sigma(i),j)^*,y\in W(j)
\end{array}$$
where $B_{ij}^{\text{t}}$ is the nondegenerated pairing $\mathcal{Q}(i,j)^*\otimes \mathcal{Q}(\sigma(i),j)^*\to \mathbb{C}$ induced by $B_{ij}$. 
\end{proof}

Now we will define an involution over $\mathfrak{L}(\ww)$. 
To ensure the well-definedness, we need the following lemma. 

\begin{Lemma} For any subrepresentation $F$ of $K_RW|_{\mathcal{Q}}$, 
the annihilator $F^\perp$ is also a subreprsentation of $K_RW|_{\mathcal{Q}}$. 
\end{Lemma}
\begin{proof}
Note that $h(F(i))\subseteq F(j)$ if and only if $\sigma(\bar{h})(F(j)^\perp)\subseteq F(i)^\perp$. 
\end{proof}

Thanks to the lemma above, it is safe to define 
\begin{equation}\label{involutionsigma}
\sigma: \mathfrak{L}(\ww)\longrightarrow \mathfrak{L}(\ww)
\end{equation}
the involution sending $F$ to $F^\perp$. 
Define the \emph{$\sigma$-quiver variety} $\mathfrak{R}(\ww)$ to be the fixed loci of fixed points of $\mathfrak{L}(\ww)$ under this involution. 

\begin{Rmk}\label{Rmkonsigmavariety}
Actually, the Nakajima quiver variety $\mathfrak{L}(\ww)$ is a closed subvariety of the Nakajima quiver variety $\mathfrak{M}(\ww)$. 
To be exact, it is the fibre of $0\in \mathfrak{M}_0(\ww)$ under a proper morphism $\mathfrak{M}(\ww)\to \mathfrak{M}_0(\ww)$. 
Our definition of $\mathfrak{R}(\ww)$ is inspired by \cite{Li18}. 
Actually, in our case (type AIII), the variety $\mathfrak{R}(\ww)$ is the intersection of $\mathfrak{L}(\ww)$ and the $\sigma$-quiver variety $\mathfrak{R}(\ww)$ in $\mathfrak{M}(\ww)$. 
It follows from the next theorem \ref{sigmacoincides}. 
In particular, $\mathfrak{R}(\ww)$ gives a family of \emph{Spaltenstein varieties}, see \cite{Li18}. 

We will not use this fact in the rest of this paper. 
\end{Rmk}

\begin{Th}\label{sigmacoincides}In type AIII, our involution $\sigma$ defined in (\ref{involutionsigma}) coincides with the involution defined in \cite{Li18} after restricting to $\mathfrak{L}(\ww)$. 
\end{Th}
\begin{proof}
We will write $\mathfrak{L}(W)$ rather than $\mathfrak{L}(\ww)$ for an $\mathbb{I}$-graded vector space $W$ to emphasize the underlying frame space. 

Firstly, the involution defined in \cite{Li18} is based on reflection functors introduced in \cite{Lusztig00Quiver} and \cite{Andrea00Aremark} independent, see also \cite{Nakajima00Reflection}. 
We will only use the reflection functor corresponding to the longest element of Weyl group. It is usually complicated to describe, but over $\mathfrak{L}(\ww)$, it has some explicite description. 

Recall the Lusztig's new symmetry introduced in \cite{Lusztig00Remarks}, 
it is an isomorphism of varieties given by 
\begin{equation}\label{Lusztignewsymmetry}
\mathfrak{L}(W)\longrightarrow \mathfrak{L}(W^*). 
\end{equation}
It is given by taking the annihilator of $K_RW$. 
In \cite{Nakajima00Reflection}, Nakajima proved this coincides with the composition of the reflection functor corresponding to the longest element of the Weyl group and the isomorphism of taking the dual representation. 
Actually, the involution defined in \cite{Li18} is the composition of the longest element of the Weyl group and the isomorphism of taking the annihilator. In particular, it is essentially (\ref{Lusztignewsymmetry}) by identifying $W^*$ and $W$ by a bilinear form. 

Assume we have an isomorphism $W^*\to W$ by a bilinear form, then it induces an isomorphism, 
\begin{equation}
K_R W^*\longrightarrow K_R W
\end{equation}
since the construction of $K_R$ is functorial. 
By the construction of Lusztig's new symmetry, after the above identification, the involution $\sigma$ introduced in \cite{Li18} coincides with ours. 
\end{proof}

\begin{Eg}\label{ExampleofsigmaNakQV}
In the Example \ref{ExampleofNakQV}, assume $n$ is an even number, and $W$ is equipped with a symplectic form $\omega$. 
The pairing over $K_R(\ww)$ is given by
$$\omega((v_i),(v'_i))=\sum_i \omega(v_i,v'_{\sigma(i)}).$$
Thus in particular, for a partial flag 
$$0=V_0\subseteq \cdots \subseteq V_i \subseteq \cdots \subseteq V_n=W$$ 
of $W$, its image under $\sigma$ is exactly 
$$0=V_n^\perp\subseteq \cdots \subseteq V_{n-i}^\perp\subseteq \cdots \subseteq V_0^\perp=W.$$
In particular, the $\sigma$-quiver variety is exactly the partial flags of type $C$ studied in \cite{BaoKujwaLiWang}. 
\end{Eg}

\begin{Eg}\label{iSmallEgA}Let us analyse the Example \ref{SmallEgA} when $n=2$.
Let us equip $W$ with a symplectic form $\omega$. 
In this case, 
$$\omega\bigg(v_1\,{\displaystyle v_2^+\atop v_2^-}\,v_3\,\,,\,\,v'_1\,{\displaystyle v^{\prime+}_2\atop v_2^{\prime-}}\,v'_3\bigg)=\omega(v_1,v'_3)+\omega(v_3,v'_1)+\omega(v_2^+,v_2^{\prime -})+\omega(v_2^-,v_2^{\prime +}).$$
As a result, we are finding subspaces $V_1,V_3$ in $W$ with $V_1^\perp =V_3$ and $V_2$ a Lagrangian subspace of $W\oplus W$ under above symplectic form such that 
$$\operatorname{pr}_2(V_2)\subseteq V_1\cap V_3,\qquad 
(V_1+V_3)\oplus 0\subseteq V_2,$$
where $\operatorname{pr}_2:W\oplus W\to W$ is the second projection. 
\begin{itemize}
\item When $\vv=022$ or $220$, $\mathfrak{R}(\vv,\ww)$ is a point, given by
$$\begin{array}{c}\xymatrix{
&W\\
\ar@{}[rr]^{\displaystyle W}="uW"_{\displaystyle 0}="dW"
0\ar"uW"\ar"dW";[]&\ar"uW";[u]
&W\ar"uW"\ar"dW";[],\\
}\end{array}\quad \text{or}\quad \begin{array}{c}\xymatrix{
&W\\
\ar@{}[rr]^{\displaystyle W}="uW"_{\displaystyle 0}="dW"
W\ar"uW"\ar"dW";[]&\ar"uW";[u]
&0\ar"uW"\ar"dW";[].\\
}\end{array}$$
Actually, the choice of $V_2$ can only be $\ker \operatorname{pr}_2$.
\item When $\vv=121$, it forces $V_1=V_3$ since $V_1^\perp=V_1$. 
As a result, $\mathfrak{R}(\vv,\ww)$ is the variety of one dimensional vector spaces $V_1$ and $\bar{V}_2$ with $V_1\subseteq W$ and $\bar{V}_2\subseteq (W\oplus V_1)/(V_1\oplus 0)$. Thus it is the associated projective bundle of $\mathcal{T}\oplus \mathcal{Q}$ over $\mathbb{P}^1$, with $\mathcal{T}$ and $\mathcal{Q}$ to be the tautological and quotient bundle respectively. 
\end{itemize}
\end{Eg}

\begin{Eg}\label{iSmallEgB}Let us deal with Example \ref{SmallEgB} when $n=m=2$. We take $N=M=W$ and equip with the same symplectic form $\omega$.
Now, the symplectic form is given by 
$$\omega\bigg(
\begin{array}{c}
\quad v_1^+\quad v_2^+\quad v_3^+\\
v_1^-\quad v_2^-\quad v_3^-\quad \\
\end{array}\,\,,\,\,
\begin{array}{c}
\quad v_1^{\prime+}\quad v_2^{\prime+}\quad v_3^{\prime+}\\
v_1^{\prime-}\quad v_2^{\prime-}\quad v_3^{\prime-}\quad \\
\end{array}
\bigg)=\sum_{\begin{subarray}{c}
\pm\in \{+,-\}\\
\bullet=1,2,3
\end{subarray}} \omega(v_\bullet^{\pm},v_{4-\bullet}^{\pm}).$$
We take the following notations
$$\xymatrix{
\fbox{$W$}&&\fbox{$W$}\\
V_1\ar[u]\ar@<+0.5ex>[r]&\ar@<+0.5ex>[l]
V_2\ar@<+0.5ex>[r]&\ar@<+0.5ex>[l]
V_3\ar[u].}$$
We have 
\begin{itemize}
\item For $\vv=024$ and $\vv=402$, the sigma quiver variety is empty. 
This is because $\operatorname{pr}_\bullet(V_2)$ is always nonzero for $\bullet=1,2$, since $\ker\operatorname{pr}_\bullet$ is not Lagrangian. 
\item 
For $\vv=123$ or $\vv=321$, the representation in $\mathfrak{R}(\vv,\ww)$ takes the form 
$$\begin{array}{c}\xymatrix{
\ar@{}[r]|>>{\displaystyle W}="tN"
&&&\ar@{}[r]|<<{\displaystyle W}="tM"
&\\
\ar@{}[r]_>>{\displaystyle 1}="aN"
&\ar@{}[r]^<<{\displaystyle 0}="aM"
\ar@{}[r]_>>{\displaystyle 1}="bN"
&\ar@{}[r]^<<{\displaystyle 1}="bM"
\ar@{}[r]_>>{\displaystyle 1}="cN"
&\ar@{}[r]^<<{\displaystyle 2,}="cM"&
\ar"aN";"tN"\ar"bN";"aN"\ar"cN";"bN"
\ar"cM";"tM"\ar"aM";"bM"\ar"bM";"cM"
}\end{array}\quad \text{or}\quad
\begin{array}{c}\xymatrix{
\ar@{}[r]|>>{\displaystyle W}="tN"
&&&\ar@{}[r]|<<{\displaystyle W}="tM"
&\\
\ar@{}[r]_>>{\displaystyle 2}="aN"
&\ar@{}[r]^<<{\displaystyle 1}="aM"
\ar@{}[r]_>>{\displaystyle 1}="bN"
&\ar@{}[r]^<<{\displaystyle 1}="bM"
\ar@{}[r]_>>{\displaystyle 0}="cN"
&\ar@{}[r]^<<{\displaystyle 1.}="cM"&
\ar"aN";"tN"\ar"bN";"aN"\ar"cN";"bN"
\ar"cM";"tM"\ar"aM";"bM"\ar"bM";"cM"
}\end{array}$$
Here the numbers indicate the dimension. 
As a result, $\mathfrak{R}(\vv,\ww)$ are both isomorphic to $\mathbb{P}^1\times \mathbb{P}^1$. 
\item When $\vv=222$, it is slightly complicated. 
\begin{itemize}
\item If $\operatorname{pr}_\bullet(V_2)$ has dimension $2$ for $\bullet =1 $ or $\bullet=2$, then $V_1=0\oplus W$ and $V_3=W\oplus 0$. 
In this case the choice of $V_2$ is arbitrary. 
\item If $\operatorname{pr}_\bullet(V_2)$ are both of dimension $1$ for $\bullet=1,2$, then $V_2=V_2^+\oplus V^-_2$ for $V_2^+\subseteq W\oplus 0$ and $V_2^-\subseteq 0\oplus W$. 
We need to require 
$0\oplus V_2^- \subseteq V_1\subseteq 
V_2^+\oplus W$. 
\end{itemize}
As a result, $\mathfrak{R}(\vv,\ww)$ is a union of two irreducible components $\Sigma_1$ and $\Sigma_2$ with $\Sigma_1$ isomorphic to the Lagrangian Grassmannian of $W\oplus W$ and $\Sigma_2$ isomorphic to the associated Grassmannian of $\mathcal{Q}_-\oplus \mathcal{T}_+$ over $\mathbb{P}^1\times \mathbb{P}^1$ with $\mathcal{T}^{+}$ the tautological bundle of the first factor, and $\mathcal{Q}^{-}$ the quotient bundle of the second factor. Their intersection is $\mathbb{P}^1\times \mathbb{P}^1$. 

\end{itemize}
\end{Eg}

\begin{Prop}\label{epsilonphiduality} We have 
\begin{equation}
\epsilon_i(F)=\phi_{\sigma(i)}(F^\perp),\quad 
\phi_i(F)=\epsilon_{\sigma(i)}(F^\perp).
\end{equation}
\end{Prop}

\begin{proof}
Note that $\sigma$ sets an isomorphism between $\rho_{i1}^{-1}(F)$ and $\rho_{\sigma(i)2}^{-1}(F^\perp)$. 
Thus the results follows from Proposition \ref{Importantcomplex}, 
\end{proof}

\begin{Coro}\label{vvforsigmaquivervariety}
Assume $\ddim F=\vv$, and $\ddim F^\perp=\vv'$. 
Then $-\sigma(\ww-C\vv)=\ww-C\vv'$. 
In particular, the variety $\mathfrak{R}(\vv,\ww)$ is empty unless $\sigma(\ww-C\vv)=-(\ww-C\vv')$. 
\end{Coro}

\begin{proof}
This can be shown simply by taking difference of \ref{epsilonphiduality}, but let us prove by direct computation.
Note that by our construction,  $\ddim F+\sigma\ddim F^\perp =\vv'+\sigma\vv=\sigma\vv'+\vv=\ddim K_RW=:\vv_0$. 
By proposition \ref{dimKRW}, we have $C\vv_0=\ww+\sigma\ww$. 
Hence we have
$\ww+\sigma\ww-C(\sigma\vv+\vv')=0$. 
\end{proof}

\section{Geometric realization of $\mathbf{U}^\iota$}

Given an involution $\sigma$ over $Q$, 
it induces an involution over $\mathfrak{g}$ 
\begin{equation}
E_i\longmapsto F_{\sigma(i)},\quad 
F_i\longmapsto E_{\sigma(i)},\quad 
H_i\longmapsto -H_{\sigma(i)}.
\end{equation}
We have a Cantan decomposition 
\begin{equation}
\mathfrak{g}=\mathfrak{k}\oplus \mathfrak{p}
\end{equation}
with $\mathfrak{k}$ (resp. $\mathfrak{p}$) the subspace of $\mathfrak{g}$ the eigenspace of $\sigma$ to $1$ (resp. $-1$). 
The pair $(\mathfrak{g},\mathfrak{k})$ forms a symmetric pair, see \cite{onishchik2012lie} for more background. 
Denote 
\begin{equation}
B_i=E_i+F_{\sigma(i)},\qquad h_i=H_i-H_{\sigma(i)}. 
\end{equation}
It is easy to show that $\mathfrak{k}$ is the Lie algebra generated by $B_i$ and $h_i$ with relations
\begin{gather}
h_{i}+h_{\sigma(i)}=0,\label{sigmareltaion}\\
[h_i,h_j]=0, \quad 
[h_i,B_j]=(c_{ij}-c_{\sigma(i)j})B_j\label{Rootreltaion}\\
[B_i,B_{\sigma(i)}]=h_i\label{EFrelation}\\
[B_i,B_j]=0\quad \text{ for $\sigma(i)\neq j$ and $c_{ij}=0$}\label{SerreRelation1}\\
[B_i,[B_i,B_j]]=0\quad \text{ for $\sigma(i)\neq i$ and $c_{ij}=-1$}\label{SerreRelation2}\\
[B_i,[B_i,B_j]]=B_j \quad \text{ for $\sigma(i)=i$ with $c_{ij}=-1$}\label{iSerreRelation}
\end{gather}
Let $\mathbf{U}^{\iota}$ be the universal enveloping algebra of $\mathfrak{k}$. 
Actually, this is the classical limit of the quasi-split quantum symmetric pair of type AIII introduced in \cite{Letzter1999}. 

Analogy to Hecke correspondence $\mathfrak{P}_i$, we define the \emph{$\iota$Hecke correspondence} $\mathfrak{B}_i$ to be the variety of pairs $(F_1,F_2)\in \mathfrak{R}(\ww)\times \mathfrak{R}(\ww)$ such that
\begin{equation}
\ddim F_1/F_1\cap F_2=\mathbf{e}_{\sigma(i)}\text{ and }
\ddim F_2/F_1\cap F_2=\mathbf{e}_i.
\end{equation}
Now we have two projections 
$$\xymatrix{&\mathfrak{B}_i\ar[dr]^{\pi_{i2}}\ar[dl]_{\pi_{i1}}\\\mathfrak{R}(\ww)&&\mathfrak{R}(\ww).}$$

For $i\in \mathbb{I}$, assume $\sigma(i)\neq i$. 
Let $O(i)$ be the neighborhood of $i$, i.e. it contains all edges incident to $i$. 
Assume we are given two representations $F$ and $G$ such that $F$ and $G$ agree over the full quiver of vertices $\mathbb{I}\setminus \{i,\sigma(i)\}$. Then we can define 
$F\wr_i G$ to be the representation $H$ such that
$$H|_{O(i)}=F|_{O(i)}\cap G|_{O(i)},\,\text{and } 
H|_{O(\sigma(i))}=F|_{O(\sigma(i))}+ G|_{O(\sigma(i))}.$$
This is well-defined since there is no common edge in $O(i)$ and $O(\sigma(i))$ by our assumption.
Define $F\wr^i G=F\wr_{\sigma(i)}G$. We have the following equality of dimension vector
$$\ddim(F\wr_iG)+\ddim(F\wr^iG)=\ddim(F)+\ddim(G). $$
Moreover, if $F,G\in \mathfrak{R}(\ww)$, then so are $F\wr_iG$ and $F\wr^iG$.

\begin{Lemma}\label{reducetoHeckeCorr}For $\sigma(i)\neq i$, consider the following commutative diagram 
$$\xymatrix{
\mathfrak{B}_i\ar[r]^g\ar[d]_{\pi_{i2}}&\mathfrak{P}_i\ar[d]^{\rho_{i2}}\\
\mathfrak{R}(\ww)\ar[r]^{\subseteq} & \mathfrak{L}(\ww).}$$
where $g(F_1,F_2)=(F_1\cap F_2,F_2)$. 
Then $g$ is an isomorphism onto $\rho_{i2}^{-1}(\mathfrak{R}(\ww))$. 
\end{Lemma}
\begin{proof}
The inverse is given by $(F_0,F_2)\mapsto (F_0\wr_i F_0^\perp,F_2)$. 
\end{proof} 

\begin{Coro}\label{computationoffibre1}
For $F\in \mathfrak{R}(\vv,\ww)$, if $\sigma(i)\neq i$, 
\begin{gather}
\chi(\pi_{i1}^{-1}(F))=\phi_i(F),\\
\chi(\pi_{i2}^{-1}(F))=\epsilon_i(F).
\end{gather}
\end{Coro}

Now turn to the case $\sigma(i)=i$. 
Note that $(\ww-C\vv)_i=0$ by Proposition \ref{vvforsigmaquivervariety}. In particular $\epsilon_i(F)=\phi_i(F)$ for $F\in \mathfrak{R}(\vv,\ww)$. 
We need the following linear algebra. 

\begin{Lemma}\label{SymplecticLineaAlgebra1}
Let $V$ be a symplectic vector space of dimension $2n$. 
For a given Lagrangian subspace $L_0$, the variety of Lagrangian subspace $L$ such that $\dim L\cap L_0=n-1$ has Euler characteristic $n$. 
\end{Lemma}
\begin{proof}
Let this variety be $\Sigma$. Then we have a morphism $\Sigma=\mathbb{P}(L^*)$ by sending $L$ to $L\cap L_0$. 
This is a fibre bundle with fibre at $L'$ the variety of Lagrangian subspace between $L'$ and $L^{\prime\perp}$ other than $L_0$. The fibre is just $\mathbb{C}P^1\setminus \mathsf{pt}$ whose Euler characteristic is $1$. 
\end{proof}

\begin{Coro}\label{computationoffibre2}
For $F\in \mathfrak{R}(\vv,\ww)$, if $\sigma(i)= i$, 
\begin{equation}
\chi(\pi_{i1}^{-1}(F))=\phi_i(F)=
\chi(\pi_{i2}^{-1}(F))=\epsilon_i(F).
\end{equation}
\end{Coro}
\begin{proof}
It is clear that $\phi_i(F)$ is the variety of Lagrangian subspaces of $K_RW(i)$ containing $\sum_{h:j\to i} h(F(j))$ other than $F(i)$. 
By the Lemma \ref{SymplecticLineaAlgebra1} above, it is exactly $\epsilon_i(F)$.
\end{proof} 

The following theorem is the main theorem of this paper. 

\begin{Th}\label{MainTh}
For type AIII,
there is a well-defined action of $\mathbf{U}^{\iota}$-action over $\Fun(\mathfrak{R}(\ww))$ by 
\begin{gather}
h_i\longmapsto (\ww-C\vv)_i\\
B_i\longmapsto \pi_{i1*}\pi_{i2}^*
\end{gather}
\end{Th}

We will prove this theorem in the rest of paper. 

\begin{Eg}In the case of the Example \ref{ExampleofsigmaNakQV},
this is the result of \cite{BaoKujwaLiWang} at $q=1$. 
\end{Eg}

\begin{Rmk}
Let us denote $\overline{\mathfrak{B}}_i$ be the variety of pairs $(F_1,F_2)\in \mathfrak{R}(\ww)\times \mathfrak{R}(\ww)$ such that there exists an $F_0\in \mathfrak{L}(\ww)$ with
\begin{equation}
\begin{array}{c}
F_0\subseteq F_1\\
\ddim F_1/F_0=\mathbf{e}_{\sigma(i)}
\end{array}\text{ and }
\begin{array}{c}
F_0\subseteq F_2\\
\ddim F_2/F_0=\mathbf{e}_i.
\end{array}
\end{equation}
Note that $\overline{\mathfrak{B}_i}=\mathfrak{B}_i$ if $\sigma(i)\neq i$, $\overline{\mathfrak{B}_i}=\mathfrak{B}_i\cup \Delta\mathfrak{R}(\ww)$ if $\sigma(i)=i$ where $\Delta$ is the diagonal. 
This also defines a well-defined action of $\mathbf{U}^\iota$ in a similar manner. 
This corresponds to the correspondence used in \cite{BaoKujwaLiWang}. This is because $B_i\longmapsto B_i+\delta_{i=\sigma(i)}$ is an automorphism of $\mathbf{U}^\iota$. 

Note that $\overline{\mathfrak{B}_i}$ is closed, since $\bar{\pi}_{i\bullet}$ is proper by the proof of Lemma \ref{SymplecticLineaAlgebra1} and Corollary \ref{computationoffibre2}. 
Here we take this convention since it is parallel to the geometric realization of Hecke algebras where the correspondences are not closed.
\end{Rmk}

\section{Relations (I)}

In this section, we will prove the relations (\ref{sigmareltaion}),  (\ref{EFrelation}) and (\ref{SerreRelation1}). 
They are parallel to what was done in \cite{Nakajima94}. 
Especially, the proof of the last relations (\ref{SerreRelation1}), known as \emph{Serre relations}, was originally due to Lusztig \cite{Lusztig1991}. 

\begin{Prop}The relations (\ref{sigmareltaion}) and (\ref{Rootreltaion}) hold. 
\end{Prop}
\begin{proof}
The relation (\ref{sigmareltaion}) follows  from Proposition \ref{vvforsigmaquivervariety} that $(\ww-C\vv)+\sigma(\ww-C\vv)=0$. 
For $(F_1,F_2)\in \mathfrak{B}_i$, we have 
$$(\ww-C\ddim F_1)_j-(\ww-C\ddim F_2)_j=(C(\mathbf{e}_i-\mathbf{e}_{\sigma(i)}))_j=c_{ij}-c_{\sigma(i)j}.$$
Thus the relation (\ref{Rootreltaion}) holds.
\end{proof}

\begin{Prop} The relation (\ref{EFrelation}) holds. 
\end{Prop}

\begin{proof}
It suffices to show when $\sigma(i)\neq i$. 
Let $\mathfrak{B}_i\mathfrak{B}_{\sigma(i)}=
\mathfrak{B}_i\times_{\mathfrak{R}(\ww)}\mathfrak{B}_{\sigma(i)}$ be the fibre product. Consider the following diagram 
$$\xymatrix{
&\mathfrak{B}_i\mathfrak{B}_{\sigma(i)}\ar[dl]_{p_1}\ar[dr]^{p_3}\ar[d]^f\\
\mathfrak{R}(\ww)&\mathfrak{R}(\ww)\times \mathfrak{R}(\ww)\ar[r]_{q_3}\ar[l]^{q_1}&\mathfrak{R}(\ww).}$$
Let $\mathfrak{S}_i$ be the preimage of the diagonal $\Delta\mathfrak{R}(\ww)\subseteq \mathfrak{R}(\ww)\times \mathfrak{R}(\ww)$ under $f$ and $\mathfrak{A}_i$ its complement. 
By (\ref{Cartesiansquare}), and (\ref{additionprinciple})
for any $\varphi\in \Fun(\mathfrak{R}(\ww))$, 
$$B_iB_{\sigma(i)}\varphi = p_{1*}p_3^*\varphi 
= q_{1*}\big(f_*\mathbf{1}_{\mathfrak{S}_i}\cdot q^*_3\varphi\big)+
q_{1*}\big(f_*\mathbf{1}_{\mathfrak{A}_i}\cdot q^*_3\varphi\big).
$$
Note that $\mathfrak{S}_i$ is isomorphic to $\mathfrak{B}_i$, and thus $q_{1*}\big(f_*\mathbf{1}_{\mathfrak{S}_i}\cdot q^*_3\varphi\big)=\pi_{i1*}\pi_{i1}^*\varphi$. 
We will use the same notation $f$ to denote the counterpart of $\sigma(i)$ and we have similarly 
$q_{1*}\big(f_*\mathbf{1}_{\mathfrak{S}_{\sigma(i)}}\cdot q^*_3\varphi\big)=\pi_{i2*}\pi_{i2}^*\varphi$. 
By Corollary \ref{computationoffibre1}, we have 
$$
\pi_{i1*}\pi_{i1}^*\varphi-\pi_{i2*}\pi_{i2}^*\varphi=(\ww-C\vv)_i\varphi=h_i\varphi.
$$
It rests to show $f_*\mathbf{1}_{\mathfrak{A}_i}=f_*\mathbf{1}_{\mathfrak{A}_{\sigma(i)}}$. 
By definition, the fibre of $f$ at $(F_1,F_2)$ when $F_1\neq F_2$ in $\mathfrak{A}_i$ (resp $\mathfrak{A}_{\sigma(i)}$) is the point $F_1\wr^i F_2$ (resp $F_1\wr_iF_2$) if it has correct dimension. 
This shows $f_*\mathbf{1}_{\mathfrak{A}_i}=f_*\mathbf{1}_{\mathfrak{A}_{\sigma(i)}}$. 
\end{proof} 

\begin{Prop}The relation (\ref{SerreRelation1}) holds. 
\end{Prop}

\begin{proof}
Let $\mathfrak{K}$ be the variety of pair $(F_1,F_2)\in \mathfrak{R}(\ww)\times \mathfrak{R}(\ww)$ such that 
$$\ddim F_1/F_1\cap F_2=\mathbf{e}_{\sigma(i)}+\mathbf{e}_{\sigma(j)},\text{ and }
\ddim F_2/F_1\cap F_2=\mathbf{e}_i+\mathbf{e}_{j}.$$
We have two morphisms
$$\mathfrak{B}_i\times_{\mathfrak{R}(\ww)}\mathfrak{B}_j\longrightarrow \mathfrak{K}\longleftarrow \mathfrak{B}_j\times_{\mathfrak{R}(\ww)}\mathfrak{B}_i.$$
We will show both of them are isomorphisms for $\sigma(i)\neq j$. 
If $c_{\sigma(i)j}=0$, then $O(i)\cup O(\sigma(i))$ and $O(j)\cup O(\sigma(j))$ has no common edge, thus the above two morphisms are both isomorphisms. 
If $c_{\sigma(i)j}=-1$, $\sigma(j)\neq j$ and $\sigma(i)\neq i$. The relation of $(F_1,F_2)\in \mathfrak{K}$ can be illustrated as follows, 
$$\xymatrix{
\cdots\ar@{=}[d]
\ar@<+0.5ex>[r]&\ar@<+0.5ex>[l]
F_1(i)\ar@{..}[d]|{\cap\shortmid}
\ar@<+0.5ex>[r]&\ar@<+0.5ex>[l]
F_1(\sigma(j))\ar@{..}[d]|{\cup\shortmid}
\ar@<+0.5ex>[r]&\ar@<+0.5ex>[l]
\cdots\ar@{=}[d]
\ar@<+0.5ex>[r]&\ar@<+0.5ex>[l]
F_1(j)\ar@{..}[d]|{\cap\shortmid}
\ar@<+0.5ex>[r]&\ar@<+0.5ex>[l]
F_1(\sigma(i))\ar@{..}[d]|{\cup\shortmid}
\ar@<+0.5ex>[r]&\ar@<+0.5ex>[l]
\cdots\ar@{=}[d]\\
\cdots
\ar@<+0.5ex>[r]&\ar@<+0.5ex>[l]
F_2(i)
\ar@<+0.5ex>[r]&\ar@<+0.5ex>[l]
F_2(\sigma(j))
\ar@<+0.5ex>[r]&\ar@<+0.5ex>[l]
\cdots
\ar@<+0.5ex>[r]&\ar@<+0.5ex>[l]
F_2(j)
\ar@<+0.5ex>[r]&\ar@<+0.5ex>[l]
F_2(\sigma(i))
\ar@<+0.5ex>[r]&\ar@<+0.5ex>[l]
\cdots}$$
It is clear that $h(F_1(i))\subseteq F_2(\sigma(j))$ for $h:i\to \sigma(j)$ and so on. In particular, 
$$\cdots\oplus F_1(i)\oplus F_2(\sigma(j))\oplus \cdots \oplus F_2(j)\oplus F_1(\sigma(i))\oplus \cdots $$
is a well-defined representation which is the unique representation in the fibre of the left morphism. 
The same argument for the right morphism shows the equality. 
\end{proof} 

\begin{Prop}The relation (\ref{SerreRelation2}) holds. 
\end{Prop}

\begin{proof}
We will show the case $\sigma(j)=j$, and the case $\sigma(j)\neq j$ will be left to readers.
Construct the following diagram 
$$\xymatrix{
&\mathfrak{B}_i\mathfrak{B}_i\mathfrak{B}_j\ar[d]_{f}\ar[dr]\ar[dl]\\
\mathfrak{R}(\ww)&\mathfrak{R}(\ww)\times \mathfrak{R}(\ww)\ar[r]_{q_2}\ar[l]^{q_1}&\mathfrak{R}(\ww),}$$
where 
$$\mathfrak{B}_i\mathfrak{B}_i\mathfrak{B}_j=\mathfrak{B}_i\times_{\mathfrak{R}(\ww)}\mathfrak{B}_i\times_{\mathfrak{R}(\ww)}\mathfrak{B}_j.$$
By (\ref{Cartesiansquare}), for any $\varphi\in \Fun(\mathfrak{R}(\ww))$, 
$$B_iB_iB_j=q_{1*}(f_*\mathbf{1}_{\mathfrak{B}_i\mathfrak{B}_i\mathfrak{B}_j}\cdot q_2^*\varphi).$$
We will abuse of notation to denote $f$ the counterpart of $\mathfrak{B}_i\mathfrak{B}_j\mathfrak{B}_i$ or 
$\mathfrak{B}_j\mathfrak{B}_i\mathfrak{B}_i$. 
We will show that 
\begin{equation}\label{SerreRelations}
f_*\mathbf{1}_{\mathfrak{B}_i\mathfrak{B}_i\mathfrak{B}_j}-2f_*\mathbf{1}_{\mathfrak{B}_i\mathfrak{B}_j\mathfrak{B}_i}+
f_*\mathbf{1}_{\mathfrak{B}_j\mathfrak{B}_i\mathfrak{B}_i}=0.
\end{equation}
Let $h$ the arrow from $i$ to $j$. 
For $(F_1,F_2)\in \mathfrak{R}(\ww)\times \mathfrak{R}(\ww)$ with 
$$\dim F_1/F_1\cap F_2=2\mathbf{e}_{\sigma(i)}+\mathbf{e}_j,\text{ and }
\dim F_2/F_1\cap F_2=2\mathbf{e}_{i}+\mathbf{e}_j,$$
we will compute (\ref{SerreRelations}). We can illustrate the relation of $(F_1,F_2)$ as follows
$$\xymatrix{
\cdots\ar@{=}[d]
\ar@<+0.5ex>[r]&\ar@<+0.5ex>[l]
F_1(i) \ar@{..}[d]|{\cap\shortmid}
\ar@<+0.5ex>[r]^h&\ar@<+0.5ex>[l]^{\bar{h}}
F_1(j) \ar@{..}[d]
\ar@<+0.5ex>[r]&\ar@<+0.5ex>[l]
F_1(\sigma(i)) \ar@{..}[d]|{\cup\shortmid}
\ar@<+0.5ex>[r]&\ar@<+0.5ex>[l]
\cdots\ar@{=}[d]\\
\cdots
\ar@<+0.5ex>[r]&\ar@<+0.5ex>[l]
F_2(i) 
\ar@<+0.5ex>[r]^h&\ar@<+0.5ex>[l]^{\bar{h}}
F_2(j)
\ar@<+0.5ex>[r]&\ar@<+0.5ex>[l]
F_2(\sigma(i))
\ar@<+0.5ex>[r]&\ar@<+0.5ex>[l]
\cdots\\
}$$
We have 
\begin{itemize}
\item The fibre of $f$ in $\mathfrak{B}_j\mathfrak{B}_i\mathfrak{B}_i$
is exactly the variety of subspaces $X$ of correct dimension between $F_1(i)$ and $F_2(i)$ if $\bar{h}(F_2(j))\subseteq F_1(i)$, otherwise, it is empty. 
\item The fibre of $f$ in 
$\mathfrak{B}_i\mathfrak{B}_j\mathfrak{B}_i$
is exactly the variety of subspaces $X$ of correct dimension between $F_1(i)$ and $F_2(i)$ of correct dimension such that $$h(X)\subseteq F_1(j)\cap F_2(j),\qquad 
\bar{h}(F_1(j)+F_2(j))\subseteq X. $$
\item The fibre of $f$ in 
$\mathfrak{B}_i\mathfrak{B}_i\mathfrak{B}_j$
is exactly the variety of subspaces $X$ of correct dimension between $F_1(i)$ and $F_2(i)$ if $h(F_2(i))\subseteq F_1(j)$, otherwise, it is empty.
\end{itemize}
There are four cases. 
\begin{itemize}
\item 
If $h(F_2(i))\subseteq F_1(j)$ and $\bar{h}(F_2(j))\subseteq F_1(i)$, the three fibres are all $\mathbb{C}P^1$.
In particular, (\ref{SerreRelations}) holds. 
\item 
If $h(F_2(i))\not\subseteq F_1(j)$ and $\bar{h}(F_2(j))\subseteq F_1(i)$. 
Then the fibre of $\mathfrak{B}_i\mathfrak{B}_j\mathfrak{B}_i$, has only one point, i.e. the subspace
$$h^{-1}(F_1(j)\cap F_2(j))\cap F_2(i). $$
To prove this, it suffices to show it has a correct dimension.
Otherwise, $h^{-1}(F_1(j))\cap F_2(i)=F_1(i)$. 
Actually, we have an injection induced by $h$
$$\dfrac{F_2(i)}{h^{-1}(F_1(j)\cap F_2(j))\cap F_2(i)}\longrightarrow \dfrac{F_2(j)}{F_1(j)\cap F_2(j)}.$$
This is a contradiction. 
As a result, the fibres are $\mathbb{C}P^1$, $\mathsf{pt}$ and empty respectively. 
\item If $h(F_2(i))\subseteq F_1(j)$ and $\bar{h}(F_2(j))\not\subseteq F_1(i)$. 
Then the fibre of $\mathfrak{B}_i\mathfrak{B}_j\mathfrak{B}_i$, has only one point, i.e. the subspace
$$\bar{h}(F_1(j)+ F_2(j))+ F_1(i). $$
It suffices to show it has a correct dimension.
Otherwise, $\bar{h}(F_1(j)+ F_2(j))+ F_1(i) = F_2(i)$. 
Actually, we have a surjection induced by $\bar{h}$ 
$$\dfrac{F_2(j)}{F_1(j)\cap F_2(j)}=\dfrac{F_1(j)+F_2(j)}{F_1(j)}\longrightarrow \frac{\bar{h}(F_1(j)+ F_2(j))+ F_1(i)}{F_1(i)}. $$
The is a contradiction. 
As a result, the fibres are empty, $\mathsf{pt}$ and $\mathbb{C}P^1$ respectively. 
\item If $h(F_2(i))\not\subseteq F_1(j)$ and $\bar{h}(F_2(j))\not\subseteq F_1(i)$. 
From the discussion above, if the fibre of  $\mathfrak{B}_i\mathfrak{B}_j\mathfrak{B}_i$ is non-empty, then it is 
$$h^{-1}(F_1(j)\cap F_2(j))\cap F_2(i)=\bar{h}(F_1(j)+ F_2(j))+ F_1(i). $$
Note that we have the following complex (see \cite{Lusztig1991})
$$\dfrac{F_2(i)}{h^{-1}(F_1(j)\cap F_2(j))\cap F_2(i)}\longrightarrow \dfrac{F_2(j)}{F_1(j)\cap F_2(j)}\longrightarrow \frac{\bar{h}(F_1(j)+ F_2(j))+ F_1(i)}{F_1(i)}$$
which is injective on the left and surjective on the right. 
But by dimension reason, it is impossible. 
As a result, the fibres are all empty. 
\end{itemize}
In particular, the relation (\ref{iSerreRelationGeo}) holds. 
\end{proof} 


\section{Relations (II)}

In this section, we will show the relation (\ref{iSerreRelation}) which is known as \emph{$\iota$Serre relations}. 
Assume $\sigma(i)=i$ and $c_{ij}=-1$ in this section. 
For any $F\in \mathfrak{L}(\vv,\ww)$, denote 
\begin{equation}
\mathcal{I}(F)=\sum_{h':j'\to i} h'(F(j')).
\end{equation}
Note that $\dim \mathcal{I}(F)= \vv_i-\epsilon_i(F)$ by definition. 

\begin{Lemma}\label{Difference1iSerre}
If $(F_1,F_2)\in \mathfrak{B}_j$, then 
$$\dim \mathcal{I}(F_1+ F_2)-\dim \mathcal{I}(F_1\cap F_2)=1.$$
In particular $\mathcal{I}(F_1)\subseteq \mathcal{I}(F_2)$ or 
$\mathcal{I}(F_2)\subseteq \mathcal{I}(F_1)$. 
\end{Lemma}

\begin{proof}
By Proposition \ref{epsilonphiduality}, 
$$\begin{array}{rl}
\dim \mathcal{I}(F_1+ F_2)-\dim \mathcal{I}(F_1\cap F_2)
& = \epsilon_i(F_1\cap F_2)-\epsilon_i(F_1+F_2)\\
& = \phi_i(F_1+F_2)-\epsilon_i(F_1+F_2)\\
& = (\ww-C\ddim (F_1+F_2))_i=-C\mathbf{e}_j=1
\end{array}$$
The equality of the last row is due to Corollary \ref{vvforsigmaquivervariety}. 
\end{proof} 


We can say more about the relations of $\mathcal{I}(F_2)$ and $\mathcal{I}(F_1)$ where we essentially use the symplectic form in Corollary \ref{symplecticformoverKRW}. 

\begin{Lemma} \label{specialintypeA}
If $(F_1,F_2)\in \mathfrak{B}_j$, then $\mathcal{I}(F_1)\neq \mathcal{I}(F_2)$ and $\mathcal{I}(F_1)\subsetneq \mathcal{I}(F_2)$ or $\mathcal{I}(F_2)\subsetneq \mathcal{I}(F_1)$. In particular, 
$$\mathcal{I}(F_1)\cap \mathcal{I}(F_2)=\mathcal{I}(F_1\cap F_2).$$
\end{Lemma}
\begin{proof}
Otherwise $\mathcal{I}(F_1)=\mathcal{I}(F_2)=\mathcal{I}(F_1+F_2)$. Let $A=\mathcal{I}(F_1+F_2)$, and $B=\mathcal{I}(F_1\cap F_2)$. 
Let $V_\bullet=F_\bullet(j)\oplus F_{\bullet}(\sigma(j))$ for $\bullet=1,2$. Let $h$ be the arrow $i\to j$. 
Denote $g:V_\bullet\to F_\bullet(i)$ the sum of $\bar{h}$ and $\sigma(\bar{h})$. 
By definition, 
$$g(V_1+V_2)\subseteq A,\quad g(V_\bullet)\not\subseteq B, \quad
g(V_1\cap V_2)\subseteq B.$$
Therefore, 
$$\bar{g}(A^\perp)\subseteq V_1\cap V_2,\quad 
\bar{g}(B^\perp) \not\subseteq V_\bullet,\quad
\bar{g}(B^\perp) \subseteq V_1+ V_2.$$
Pick any $x\in B^\perp\setminus A^\perp$. 
By definition 
$$\bar{g}(x)\in V_1+V_2,\quad \bar{g}(x)\notin V_\bullet.$$
Note that $\bar{g}(x)=(h(x),\sigma(h)(x))\in K_RW(j)\oplus K_RW(\sigma(j))$. 
Denote $(u,v)=\bar{g}(x)=(h(x),\sigma(h)(x))\in K_RW(j)\oplus K_RW(\sigma(j))$. 
The condition is equivalent to say 
$$u\in F_2(j),v\in F_1(\sigma(j)) \text{ and } u\notin F_1(j),v\notin F_2(j).$$
In particular $F_1(j)+\mathbb{C}u=F_2(j)$ and $F_2(\sigma(j))+\mathbb{C}v=F_1(j)$. 
In particular, $\omega(u,v)\neq 0$. 
But 
$$\omega(u,v)=\omega(h(x),\sigma(h)(x))=\omega(x,\sigma(\bar{h}h)x)=\omega(x,\bar{h}hx)=\omega(\sigma(h)(x),h(x))=\omega(v,u)$$
a contradiction. 
\end{proof} 

\begin{Eg}We can examine the examples before. 
\begin{itemize}
\item In Example \ref{ExampleofsigmaNakQV}, 
$\mathcal{I}$ is exactly $V_{n/2-1}$. 
\item In Example \ref{iSmallEgA}, consider the cases of $121$ and $022$, we see that $\mathcal{I}=V_1$ and $\mathcal{I}=W$ respectively. 
\item In Example \ref{iSmallEgB}, we see $\dim \mathcal{I}=1$ for $\vv=321$. For $\vv=222$, if $F$ corresponds to a point over Langrangian Grassmannian, then $\mathcal{I}=0$. 
If $F$ corresponds a point out of it, then  $\operatorname{pr}_1(V_1)\neq 0$ and similarly
$\operatorname{pr}_2(V_1^\perp)\neq 0$. 
So $\dim \mathcal{I}=2$. 
\end{itemize}
As a result, we cannot predict which $\mathcal{I}(F_\bullet)$ is bigger for $(F_1,F_2)\in \mathfrak{B}_j$ from the dimension vector. 
\end{Eg}

Construct the following diagram 
$$\xymatrix{
&\mathfrak{B}_i\mathfrak{B}_i\mathfrak{B}_j\ar[d]_{f}\ar[dr]\ar[dl]\\
\mathfrak{R}(\ww)&\mathfrak{R}(\ww)\times \mathfrak{R}(\ww)\ar[r]_{q_2}\ar[l]^{q_1}&\mathfrak{R}(\ww),}$$
where 
\begin{equation}
\mathfrak{B}_i\mathfrak{B}_i\mathfrak{B}_j
=\mathfrak{B}_i\times_{\mathfrak{R}(\ww)}\mathfrak{B}_i\times_{\mathfrak{R}(\ww)}\mathfrak{B}_j.
\end{equation}
We will abuse of notation to denote $f$ the counterpart of $\mathfrak{B}_i\mathfrak{B}_j\mathfrak{B}_i$,  
$\mathfrak{B}_j\mathfrak{B}_i\mathfrak{B}_i$. 
By (\ref{Cartesiansquare}), it suffices to show 
\begin{equation}
f_*\mathbf{1}_{\mathfrak{B}_i\mathfrak{B}_i\mathfrak{B}_j}-2f_*\mathbf{1}_{\mathfrak{B}_i\mathfrak{B}_j\mathfrak{B}_i}+
f_*\mathbf{1}_{\mathfrak{B}_j\mathfrak{B}_i\mathfrak{B}_i}=\mathbf{1}_{\mathfrak{B}_j}.
\label{iSerreRelationGeo}
\end{equation}
From the next lemma, we see that it suffices to check over $(F_1,F_2)\in \mathfrak{R}(\ww)\times \mathfrak{R}(\ww)$ with 
$\dim (F_1(i)\cap F_2(i))\leq \dim F_1(i)-2$. 

\begin{Lemma}\label{SymplecticLineaAlgebra2} 
Let $V$ be a symplectiv vector space of dimension $2n$. 
For three Lagrangian subspaces $L_1,L_2,L_3$, 
denote $L_{ij}=L_i\cap L_j$ for $i,j\in \{1,2,3\}$, and $L_{123}=\cap L_i$. 
Assume $\dim L_{12}=\dim L_{23} = n-1$,
then 
\begin{equation}
\begin{array}{c}
\dim L_{13}=n-2\\
L_{123}=L_{13}
\end{array}\quad \text{ or }\quad 
\begin{array}{c}
\dim L_{13}= n-1\\
L_{12}=L_{13}=L_{23}
\end{array}\quad \text{ or }\quad 
L_1=L_3.
\end{equation}
\end{Lemma}

\begin{proof}
Assume $L_1\neq L_3$. 
Since $\dim L_{123}\geq \dim L_{12} +\dim L_{23}-\dim (L_{12}+L_{23})\geq 
\dim L_{12} +\dim L_{23}-\dim L_{2}=
n-2$, $\dim L_{123}\geq n-2$. 
Assume $L_{12}\neq L_{23}$, then $\dim L_{123}=n-2$ and $L_{12}+L_{23}=L_2$. 
Without loss of generality, we can assume $L_{123}=0$ i.e. $n=2$. 
Assume $L_{12}=\mathbb{C}x$, and $L_{23}=\mathbb{C}y$. 
Then $L_{2}=L_{12}+\mathbb{C}y$ and $L_1=L_{23}+\mathbb{C}x$. Assume further that $L_1=L_{12}+\mathbb{C}z$ and $L_3=L_{23}+\mathbb{C}w$. We will show $x,y,z,w$ are linearly independent, then $L_1\cap L_3=L_{123}=0$.

Firstly, since $L_1,L_2,L_3$ are Lagrangian, 
$\omega(x,y)=\omega(x,z)=\omega(y,w)=0$.
Secondly, since Langriang is maximal isotropic subspaces, $\omega(z,y)\neq 0\neq \omega(x,w)$. 
Thus
$$\left(\begin{matrix}
\omega(x,x)&\omega(x,y)&\omega(x,z)&\omega(x,w)\\
\omega(y,x)&\omega(y,y)&\omega(y,z)&\omega(y,w)\\
\omega(z,x)&\omega(z,y)&\omega(z,z)&\omega(z,w)\\
\omega(w,x)&\omega(w,y)&\omega(w,z)&\omega(w,w)
\end{matrix}\right)\in 
\left(\begin{matrix}
0&0&0&\mathbb{C}^{\times}\\
0&0&\mathbb{C}^{\times}&0\\
0&\mathbb{C}^\times&0&\omega(z,w)\\
\mathbb{C}^\times&0&\omega(w,z)&0
\end{matrix}\right)$$
This shows the linear independence. \end{proof}


\begin{Prop}
At the point $(F_1,F_2)$ with $F_1(i)=F_2(i)$, the relation (\ref{iSerreRelationGeo}) holds. 
\end{Prop}

\begin{proof}
The conditions of the first three fibres can be simplified as follows. 
$$\xymatrix@!=1pc{\ar@{}[dr];[r]|{X}="x"
F_1(i)\ar[r]\ar"x"\ar@{..}[d]|{\cap\shortmid}& F_1(i)
\ar@{-}"x"|{1}&F_1(\sigma(i))\ar[l]\ar"x"\ar@{..}[d]|{\cup\shortmid}\\
F_2(i)\ar[ur]&&F_2(\sigma(i))\ar[ul]
}\qquad 
\xymatrix@!=1pc{\ar@{}[dr];[r]|{X}="x"
F_1(i)\ar[r]\ar"x"\ar@{..}[d]|{\cap\shortmid}& F_1(i)
\ar@{-}"x"|{1}&F_1(\sigma(i))\ar[l]\ar"x"\ar@{..}[d]|{\cup\shortmid}\\
F_2(i)\ar[ur]\ar"x"&&F_2(\sigma(i))\ar[ul]\ar"x"
}\qquad 
\xymatrix@!=1pc{\ar@{}[dr];[r]|{X}="x"
F_1(i)\ar[r]\ar@{..}[d]|{\cap\shortmid}& F_1(i)
\ar@{-}"x"|{1}&F_1(\sigma(i))\ar[l]\ar@{..}[d]|{\cup\shortmid}\\
F_2(i)\ar[ur]\ar"x"&&F_2(\sigma(i))\ar[ul]\ar"x"
}$$
Thus when $(F_1,F_2)\notin \mathfrak{B}_j$, the fibres are all empty. So it suffices to consider the case $(F_1,F_2)\in \mathfrak{B}_j$. 
The fibres are the variety of Lagrangian subspaces $X$ of $K_RW(i)$
such that 
$$\dim(X\cap F_1(i))=\dim F_1(i)-1,$$
containing $\mathcal{I}(F_1)$, $\mathcal{I}(F_1+F_2)$ and $\mathcal{I}(F_2)$ respectively. 
By Lemma \ref{SymplecticLineaAlgebra1}, the Euler characteristic is exactly the codimension of $\mathcal{I}$ in $X$. Thus the left-hand-side of (\ref{iSerreRelationGeo}) is $\dim \mathcal{I}(F_1+F_2)-\dim\mathcal{I}(F_1\cap F_2)$ by Lemma \ref{specialintypeA}. The result follows from \ref{Difference1iSerre}. 
\end{proof}

\begin{Prop}At the point $(F_1,F_2)$ with $\dim (F_1(i)\cap F_2(i))=\dim F_1(i)-1$, the relation (\ref{iSerreRelationGeo}) holds. 
\end{Prop}

\begin{proof} 
By Lemma \ref{SymplecticLineaAlgebra2}, the fibre is empty or $\mathbb{C}P^1$ deleting two points, whose Euler characteristics are both $0$. 
\end{proof}

\begin{Lemma}\label{SymplecticLineaAlgebra3}
Let $V$ be a symplectic vector space of dimension $2n$. For two Lagangian subspaces $L_1$ and $L_2$ with $\dim (L_1\cap L_2)=n-2$. 
Then the choice of Lagrangian subspaces $L$ such that $\dim(L_1\cap V)=\dim(L_2\cap V)=n-1$ are in one-to-one correspondence to the choice of subspaces $U$ between $L_1$ and $L_1\cap L_2$ with $\dim U=n-1$. 
\end{Lemma}

\begin{proof} 
The correspondence is given by $U=L\cap L_1$ and $L_1=U+(U^{\perp}\cap L_2)$. 
It suffices to prove when $n=2$ and $L_1\cap L_2=0$. 
Assume $U=\mathbb{C}u$. Then we can find a nonzero element $v\in L_2$ such that $\omega(u,v)=0$, since $L_2$ has dimension $2$. This choice is unique up to scalar, otherwise $\omega(u,L_2)=0$, which implies $u\in L_2$. 
\end{proof}

\begin{Prop}At the point $(F_1,F_2)$ with $\dim (F_1(i)\cap F_2(i))=\dim F_1(i)-2$, the relation (\ref{iSerreRelationGeo}) holds. 
\end{Prop}
\begin{proof} By lemma \ref{SymplecticLineaAlgebra3}, the three fibres are the variety of $X$ such that 
$$F_1(i)\cap F_2(i) \subsetneq X\subsetneq F_1(i)$$
such that 
$$\begin{array}{c}
\mathcal{I}(F_1)\subseteq X,\\
\mathcal{I}(F_1)\subseteq Y; 
\end{array}\qquad 
\begin{array}{c}
\mathcal{I}(F_1)\subseteq X,\\
\mathcal{I}(F_2)\subseteq Y; 
\end{array}\qquad 
\begin{array}{c}
\mathcal{I}(F_2)\subseteq X,\\
\mathcal{I}(F_2)\subseteq Y; 
\end{array}$$
respectively, where $Y=X^\perp\cap F_2(i)$. 
Note that $X\cap Y = F_1(i)\cap F_2(i)$. 
There are four cases. 
\begin{itemize}
\item The case
$\mathcal{I}(F_1+F_2)\subseteq 
F_1(i)\cap F_2(i)$. 
Now, three fibres are the same. 
\item The case
$\mathcal{I}(F_1)\not\subseteq F_1(i)\cap F_2(i)\supseteq \mathcal{I}(F_2)$. 
The first fibre is empty. 
The second fibre is the point of the subspace $X=\mathcal{I}(F_1)+(F_1(j)\cap F_2(i))$. 
Actually, $X\subsetneq F_1(i)$ since 
$$
\dim\dfrac{\mathcal{I}(F_1)+(F_1(j)\cap F_2(i))}{F_1(j)\cap F_2(i)}
=\dim\dfrac{\mathcal{I}(F_1+F_2)+(F_1(j)\cap F_2(i))}{\mathcal{I}(F_2)+(F_1(j)\cap F_2(i))}\leq 1.$$
The last fibre is $\mathbb{C}P^1$ by Lemma \ref{SymplecticLineaAlgebra3}. 
\item The case
$\mathcal{I}(F_2)\not\subseteq F_1(i)\cap F_2(i)\supseteq \mathcal{I}(F_1)$ is similar. 
\item 
The case 
$\mathcal{I}(F_2)\not\subseteq F_1(i)\cap F_2(i)\not\supseteq \mathcal{I}(F_1)$. 
The fibres are all empty. 
The first and the last one are easy to show. 
For the second fibre, if it is not empty, then by the discussion above, the only choice is 
$X=\mathcal{I}(F_1)+(F_1(i)\cap F_2(j))$ with  
$Y=\mathcal{I}(F_2)+(F_1(i)\cap F_2(j))$. 
Note that $X$ and $Y$ do not have any inclusion relation between them, which contradicts to Lemma \ref{Difference1iSerre}. 
\end{itemize}
In particular, the relation (\ref{iSerreRelationGeo}) holds. 
\end{proof}

In all, we have the following proposition which finishes the proof of Theorem \ref{MainTh}. 

\begin{Prop} The relation (\ref{iSerreRelation}) holds. 
\end{Prop}


\section*{Declarations}

\subsection*{Conflict of Interest} The authors declare no competing interests.

\bibliography{sample}

\end{document}